\documentclass[11pt,reqno]{amsart}

\usepackage{mathtools}
\usepackage{amssymb,amsmath}
\usepackage{hyperref}
\usepackage{imakeidx}
\usepackage{tikz}
\usepackage{manfnt}
\usetikzlibrary{%
  matrix,%
  calc,%
  arrows%
}

\usepackage{graphicx}
\usepackage{fullpage}
\newcommand{\F}{\mathbb{F}}

\newcommand{\KK}{\mathbb{K}}

\newcommand{\N}{\mathbb{N}}

\newcommand{\Q}{\mathbb{Q}}
\newcommand{\R}{\mathbb{R}}
\newcommand{\RR}{\mathbb{R}}

\newcommand{\Z}{\mathbb{Z}}
\newcommand{\comment}[1]{}

\theoremstyle{definition}
\newtheorem{theorem}{Theorem}[section]

\newtheorem{lemma}[theorem]{Lemma}
\newtheorem{proposition}[theorem]{Proposition}

\theoremstyle{definition}
\newtheorem{definition}[theorem]{Definition}

\newtheorem{remark}[theorem]{Remark}
\numberwithin{equation}{subsection}
\theoremstyle{plain}
\newtheorem{thm}{Theorem}[section]
\newtheorem{prop}[theorem]{Proposition}

\newtheorem{lem}[theorem]{Lemma}

\theoremstyle{definition}

\newtheorem{rem}[theorem]{Remark}
\newtheorem{exm}[theorem]{Example}

\newtheorem*{theorem*}{Theorem}
\newtheorem*{lem*}{Lemma}
\newtheorem*{con*}{Conjecture}
\newtheorem{Claim*}{Claim}
\newtheorem*{defn*}{Definition}

\newtheorem*{rem*}{Remark}

\makeatletter
\def\imod#1{\allowbreak\mkern10mu\left({\operator@font mod}\,\,#1\right)}
\makeatother

\begin{document}

\title[SymbolicPowers 041417]{Uniform Symbolic Topologies in Normal Toric Rings}
\author{Robert M. Walker}

\address{Department of Mathematics, University of Michigan, Ann Arbor, MI, 48109}
\email{robmarsw@umich.edu}

\parskip=10pt plus 2pt minus 2pt

\begin{abstract} 
Given a normal toric algebra $R$, we compute a uniform integer $D = D(R) > 0$ such that the symbolic power $P^{(D N)} \subseteq P^N$ for all $N >0$ and all monomial primes $P$. 
We compute the multiplier $D$ explicitly in terms of the polyhedral cone data defining $R$. 
In this toric setting, we draw a connection with the F-signature of $R$ in positive characteristic. 
 \end{abstract}

\maketitle

%\tableofcontents

\section{Introduction and Conventions for the Paper}
  
Given a Noetherian commutative ring $R$, when is there an integer $D$, depending only on $R$, such that the symbolic power $P^{(D r)} \subseteq P^r$ for all prime ideals $P \subseteq R$ 
 and all positive integers $r$?\footnote{For a definition of symbolic powers of prime and radical ideals, see \cite{5authorSymbolicSurvey}. % To clarify, if $P$ is any prime ideal in a Noetherian ring $R$, its \textbf{$a$-th ($a \in \Z_{>0}$) symbolic power} ideal is the smallest $P$-primary ideal containing $P^a$: $P^{(a)} = P^a R_P \cap R := \left\lbrace f \in R \colon uf \in P^a \mbox{ for some }u \in R - P  \right\rbrace .$  %the $P$-primary component in any Lasker-Noether minimal primary decomposition of $P^a$; it is We set $P^{(0)} = P^0 = R$.  
%While $P^{(1)} = P$, the inclusion $P^{(a)} \supseteq P^a$ for $a>1$ is typically strict. %When $R$ is toric and $P$ is monomial, it suffices to work with monomials $f \in R$ and  $u \in R-P$. 
}   
In short, when does $R$ have \textbf{uniform symbolic topologies} on primes \cite{HKV2}? Moreover, can we effectively compute the multiplier $D$ in terms of simple data about $R$? In this paper, we answer this last question in the setting of torus-invariant primes in a normal toric (or semigroup) algebra. 

The Ein-Lazarsfeld-Smith Theorem  \cite{ELS}, as extended by Hochster and Huneke \cite{HH1}, says that if $R$ is a $d$-dimensional regular ring containing a field, and $D = \max \{1 , d-1 \}$,  then $Q^{(Dr)} \subseteq Q^r$ for all radical ideals $Q \subseteq R$ and all $r > 0$.\footnote{This result has been extended to all excellent regular rings, even in mixed characteristic, by Ma-Schwede \cite{MaSchwede17}.} To what extent does this theme ring true for non-regular rings? Under mild stipulations, a local ring $R$ regular on the punctured spectrum has uniform symbolic topologies on primes \cite[Cor.~3.10]{HKV}, although explicit values for $D$ remain elusive. Far less is known for rings with non-isolated singularities. 
  
This manuscript approaches the above questions for the coordinate rings of normal affine toric varieties. Such algebras are combinatorially-defined, finitely generated,   Cohen-Macaulay, normal, and have rational singularities. We deduce an explicit multiplier $D$ such that for any torus-invariant prime $P$ in $R$, 
$P^{(Dr)} \subseteq P^r \mbox{ for all }r>0;$ 
 see Theorems \ref{thm:ToricUSTPMonoPrimes001} and \ref{thm:ITAU-sandwhiching000}  for precise statements. 
Prior work towards explicit $D$ includes the  papers \cite[Table~3.3]{Walker001} and  \cite[Thm.~1.2]{Walker003} for ADE rational surface  singularities and for select domains with non-isolated singularities, respectively;  see also  %work of Dao, De Stefani, Grifo, Huneke, and N\'{u}\~{n}ez-Betancourt 
\cite[Thm.~3.29, Cor.~3.30]{5authorSymbolicSurvey} for module-finite direct summands of affine polynomial rings.

To state our main results in a special case, fix a ground field $\F$. We fix a full-dimensional pointed convex polyhedral cone $C \subseteq \R^n$ generated by %, the $\R_{\ge 0}$-linear span of  
 a finite set $G \subseteq \Z^n$, and its dual $C^\vee \subseteq \RR^n$.   
Let $R_\F$ be the semigroup algebra of the semigroup $C^\vee \cap \Z^n$. This  \textbf{toric $\F$-algebra associated to $C$} is a normal domain of finite type over $\F$ \cite[Thm.~1.3.5]{torictome} with an $\F$-basis of monomials $\{\chi^\ell \colon \ell \in C^\vee \cap \Z^n \}.$  An ideal of $R_\F$ is \textbf{monomial  (or torus-invariant)} if it is generated by a subset of these monomials. In what follows, we use $\bullet$ to denote dot products in $\RR^n$.

\begin{thm}[Cf., Theorem \ref{thm:ITAU-sandwhiching000}]\label{thm:ToricUSTPMonoPrimes001}
Suppose $C \subseteq \RR^n$ is a full-dimensional pointed polyhedral cone as above. Set $D := \max_{w \in \mathcal{B}} \left( w \bullet v_G \right)  \in  \Z_{>0}$,  where $\mathcal{B}$ is a generating set for the semigroup $C^\vee \cap \Z^n$ and $v_G \in \Z^n$ is the sum of any $($finite$)$ set $G$ of vectors in $\Z^n$ generating $C$. %the primitive generators for $C$.
Then $$P^{(Dr)} \subseteq P^{(D(r-1)+1)}  \subseteq P^{r},$$ for all $r>0$ and all monomial prime ideals $P$ in the toric ring $R_\F = \F [C^\vee \cap \Z^n]$.  
\end{thm}

\noindent We get the best result in Theorem \ref{thm:ToricUSTPMonoPrimes001} by taking $G$ to consist of the unique set of primitive generators for $C$, in which case we write $v_C$ in place of $v_G$; see Section \ref{sec:Toric-Alg-Prelims00}. For example, if $R_\F$ is the $E$-th Veronese subring of a polynomial ring of finite type over $\F$, then our $D = E$ in Theorem \ref{thm:ToricUSTPMonoPrimes001}. As another example,  for the Segre product of Veronese rings of respective degrees $E_1, \ldots, E_k$, %each in at least two variables, 
we can take $D = \sum_{i=1}^k E_i$ in Theorem \ref{thm:ToricUSTPMonoPrimes001};  see Theorem \ref{thm:SegreVeroMonoUSTP01}.

The next result covers select non-monomial primes. We note that a toric algebra satisfying the additional hypothesis below is called \textbf{simplicial}; see the discussion around Theorem \ref{thm: exact sequence}.  
\begin{thm}[{Cf., Theorem \ref{thm:height-one-primes}}]\label{thm:VarPi-FSig-Bound01} \textit{With notation as in Theorem \ref{thm:ToricUSTPMonoPrimes001}, assume moreover that %now suppose $\F$ is algebraically closed and 
 the divisor class group $\operatorname{Cl}(R_\F)$ is finite. 
 Set $U  := \operatorname{lcm} \{ \max_{w \in \mathcal{B}} (w \bullet v_C) ,  \# \operatorname{Cl}(R_\F) \}  \in \Z_{>0},$ where $\mathcal{B}$ and $v_C$ are as in Theorem \ref{thm:ToricUSTPMonoPrimes001}.   
 Then $$P^{(U (r-1) + 1)} \subseteq P^r$$ for all $r>0$, all monomial primes in $R_\F$, and all height one primes in $R_\F$.  
}\end{thm}

\noindent Tighter containments of the type $I^{(E(r-1)+1)} \subseteq I^r$, as in Theorems \ref{thm:ToricUSTPMonoPrimes001} and \ref{thm:VarPi-FSig-Bound01}  and first promoted by Harbourne, hold for  all monomial ideals in an affine polynomial ring over any field \cite[Ex.~8.4.5]{Primer}; see also recent work of Grifo-Huneke \cite{GrifoHun00}. 

At the end of Section \ref{sec:main-results}, we draw connections between the multiplier $U$ and the so-called F-signature of $R_\F$. In Section \ref{sec:Examples}, we discuss the extent to which the multipliers $D$ and $U$ in Theorems \ref{thm:ToricUSTPMonoPrimes001} and \ref{thm:VarPi-FSig-Bound01} are sharp.

\noindent \textbf{Conventions:} Throughout, $\F$ denotes an arbitrary ground field of arbitrary characteristic. All rings are commutative with identity--indeed, they are normal domains of finite type over $\F$. 

\noindent \textbf{Acknowledgements:} This paper is part of my Ph.D. thesis at the University of Michigan-Ann Arbor. My thesis adviser Karen E. Smith, along with Daniel Hern\'{a}ndez and  Jack Jeffries, suggested studying precursors for the ideals $I_\bullet (E)$ in the proof of Theorem \ref{thm:ITAU-sandwhiching000}. I am grateful for this idea.  
I acknowledge support from a NSF GRF (Grant No.  PGF-031543), NSF RTG grant DMS-0943832, and a 2017 Ford Foundation Dissertation Fellowship. 
Several computations were performed using the \texttt{Polyhedra} package in 
 Macaulay2 \cite{M2} to gain an incisive handle on the polyhedral geometry.

\section{Toric Algebra Preliminaries}\label{sec:Toric-Alg-Prelims00}

We review notation and relevant facts from toric algebra, citing Cox-Little-Schenck \cite[Ch.1,3,4]{torictome} and Fulton \cite[Ch.1,3]{introtoric}. A lattice is a free abelian group of finite rank. We 
  fix a perfect bilinear pairing $\langle \cdot , \cdot \rangle \colon M \times N \to \Z$ between two lattices $M$ and $N$;  this identifies $M$ with $\operatorname{Hom}_\Z (N, \Z)$ and  $N$ with $\operatorname{Hom}_\Z (M, \Z)$.  
Our pairing extends to a perfect pairing of finite-dimensional vector spaces $\langle \cdot , \cdot \rangle \colon M_\R \times N_\R \to \R$, where $M_\R := M \otimes_\Z \R$ and $N_\R := N \otimes_\Z \R$. % are dual. 

Fix an \textbf{$N$-rational} polyhedral cone and its $M$-rational dual: respectively, for some \textbf{finite} subset $G \subseteq N - \{0\}$ these are closed, convex sets of the form 
\begin{align*}
C &= \operatorname{Cone}(G) := \left\lbrace\sum_{v \in G} a_v \cdot v \colon \mbox{ each }a_v \in \RR_{\ge 0}\right\rbrace  \subseteq N_\R, \mbox{ and }   \\
C^\vee &:= \{w \in M_\R \colon \langle w , v \rangle \ge 0 \mbox{ for all }v \in C \} = \{w \in M_\R \colon \langle w , v \rangle \ge 0 \mbox{ for all }v \in G\}.
\end{align*}
By definition, the \textbf{dimension} of a cone in $M_\R$ or $N_\R$ is the dimension of the real vector subspace it spans; a cone is \textbf{full(-dimensional)} if it spans the full ambient space. A cone in $M_\R$ or $N_\R$ is \textbf{pointed (or strongly convex)} if it contains no line through the origin.  
A \textbf{face} of $C$ is a convex polyhedral cone $F$ in $N_\R$ obtained by intersecting $C$ with a hyperplane which is the kernel of a linear functional $m \in C^\vee$; $F$ is \textbf{proper} if $F \neq C$. 
When $C$ is both $N$-rational and pointed, so is every face $F$. Each such face $F \neq \{0\}$ has a uniquely-determined set $G_F$ of primitive generators. By definition, $v \in N$ is \textbf{primitive} if $\frac{1}{k} \cdot v \not\in N$ for all $k \in \Z_{>1}$.  

There is a bijective inclusion-reversing correspondence between faces $F$ of $C$ and faces $F^*$ of $C^\vee$, where $F^* = \{w \in C^\vee \colon \langle w , v \rangle = 0 \mbox{ for all }v \in F\}$ is the face of $C^\vee$ \textbf{dual to} $F$ \cite[Sec.~1.2]{introtoric}. Under this correspondence, 
either cone is pointed  if and only if the other is full, and   
\begin{equation}\label{eqn:face-duality-identity01}
\dim (F) + \dim (F^*) = \dim (N_\R) = \dim(M_\R).
\end{equation}

Fix an arbitrary ground field $\F$ and a cone $C$ as above in $N_\R$. The semigroup ring $R_\F = \F [C^\vee \cap M]$ is the \textbf{toric $\F$-algebra associated to $C$}. This ring $R_\F$ is a normal domain   of finite type over $\F$ \cite[Thm.~1.3.5]{torictome}. Note that $R_\F$ has an $\F$-basis $\{ \chi^m  \colon m \in  C^\vee \cap M \}$ of monomials, giving $R_\F$ an $M$-grading, where $\deg(\chi^{m}) := m$. 
A \textbf{monomial ideal (also called an $M$-homogeneous or torus-invariant ideal)} in $R_\F$ is an ideal generated by a subset of these monomials.  When $C^\vee$ is pointed, $R_\F$ also has a non-canonical $\N$-grading obtained by fixing any group homomorphism $M \to \Z$  taking positive values $C^\vee \cap M - \{0\}$. % on sends $C^\vee \cap M - \{0\}$ to $\Z_{>0}$. 
The set $\{\chi^m \colon m \in C^\vee \cap M - \{0\} \}$ generates the unique homogeneous maximal ideal $\mathfrak{m}$ under this $\N$-grading.  
 
\begin{remark}\label{rem:strong-convexity-stipulation}
In forming the toric algebra $\F [C^\vee \cap M]$, there is no loss of generality in assuming $C$ is pointed in $N_\R$. Indeed, because $C^\vee \cap M = C^\vee \cap M'$ where $M' = M \cap \{\mbox{$\R$-span of $C^\vee$ in $M_\R$}\},$ we may replace $M$ by $M'$ to assume $C^\vee$ is full in $(M')_\R$. Now, replacing $N$ and $C$ by the duals of $M'$ and $C^\vee$, we may assume that $C$ is pointed in $N' = \operatorname{Hom}_\Z (M' , \Z)$. See \cite[Thm.~1.3.5]{torictome} for details. 
\end{remark}

\noindent Fix a face $F$ of a pointed cone $C$:   \cite[p.53]{introtoric}   
 records a surjective $M$-graded ring map  
\begin{align*}
\phi_F \colon R_\F = \F[C^\vee \cap M] \twoheadrightarrow \F[F^* \cap M], \quad \phi_F (\chi^{m}) &= 
\begin{cases} \chi^{m}  &\mbox{ if $\langle m , v \rangle = 0$ for all }v \in F \\
0 &\mbox{ if $\langle m , v \rangle > 0$ for some }v \in F.   \end{cases} 
\end{align*}
Both rings are domains. The \textbf{monomial prime ideal of $F$}, $P_F : = \ker (\phi_F)$, has height equal to $\dim (F)$.     
Conversely, any monomial prime of $R_\F$ corresponds bijectively to a face of $C$. 

\begin{lem}\label{lem:prime-positivity}
Fix a face $F$ of a pointed cone $C$, and the monomial prime $P_F \subseteq R_\F$ above. Let $G_F$ be the set of primitive generators of $F$, and set $v_F : =  \sum_{v \in G_F} v \in F \cap N$.  Then 
\begin{equation}\label{eqn:mono-prime-defn}
P_F = (\{\chi^{m} \colon \mbox{$m \in C^\vee \cap M$ and the integer $\langle m , v_F  \rangle > 0$}\})R_\F.
\end{equation}
\end{lem}
\begin{proof}
First,  in defining $\phi_F (\chi^m)$ above, notice we can work with $v \in G_F$ without loss of generality. Now, fix $m \in C^\vee \cap M$. Then $\langle m , v \rangle \in \Z_{\ge 0}$ for all $v \in C \cap N$. As $\langle \cdot , \cdot   \rangle$ is bilinear, \eqref{eqn:mono-prime-defn} follows since a sum of nonnegative integers is positive if and only if one of the summands is positive.    
\end{proof}

\subsubsection{Hilbert Bases.}\label{subsub:Hilb00} First, suppose the pointed cone $C$ from Remark \ref{rem:strong-convexity-stipulation} is full.  Then there is a uniquely-determined minimal generating set $\mathcal{B}$ for $C^\vee \cap M$, in the sense that any other generating set contains $\mathcal{B}$. The set $\mathcal{B}$ is called the \textbf{Hilbert basis} of the semigroup, and consists of the \textbf{irreducible} vectors $m \in C^\vee \cap M - \{0\}$; a vector $v \in C^\vee \cap M$ is irreducible if it cannot be expressed as a sum of two vectors $m \in C^\vee \cap M - \{0\}$. 
See \cite[Prop.~1.2.17]{torictome} and \cite[Prop.~1.2.23]{torictome} for details.

\noindent In general, the pointed cone $C$ need not be full. Thus the next proposition is handy. 
\begin{prop}\label{prop: proof reduction}
Let $N'_\R$ by the $\R$-span of a pointed cone $C \subseteq N_\R$. Set $N' = N'_\R \cap N$, and consider $C$ as a full-dimensional cone in $N'_\R$ $($relabeled as $C' )$. Let $M' = \operatorname{Hom}_\Z (N', \Z)$ be the dual lattice. The toric ring $R_\F := \F [C^\vee \cap M]$ is isomorphic to $R_\F' \otimes_\F L$ where the toric ring $R_\F':= \F [(C')^\vee \cap M']$ and $L$ is a Laurent polynomial ring over $\F$. In particular, there is a bijective correspondence between the monomial primes of $R'_\F$ and $R_\F$ given by expansion and contraction of ideals along the faithfully flat ring map $\varphi \colon R'_\F \hookrightarrow R_\F' \otimes L = R_\F$.  
\end{prop}
\begin{proof}
Combine \cite[Proof of Prop.~3.3.9]{torictome} with  \cite[Discussion/Proof preceding Lem.~3.1]{Walker002}. 
\end{proof}

\subsubsection{Toric Divisor Theory}  Given a Noetherian normal domain $R$, the \textbf{divisor class group} $\operatorname{Cl}(R) = \operatorname{Cl}(\operatorname{Spec}(R))$ is the free abelian group on the set of height one prime ideals of $R$ modulo relations $ \sum_{i=1}^r a_i P_i = 0 $ when  the ideal $\bigcap_{i=1}^r P_i^{(a_i)}$ is principal.    
Note $\operatorname{Cl}(R)$ is trivial if and only if $R$ is a UFD, i.e., all height one primes in $R$ are principal. We recall the following   

\begin{lem}
[Cf., Lem.~1.1 of \cite{Walker002}] \label{thm: du Val bound 1}
When every element of $\operatorname{Cl}(R) : = \operatorname{Cl(Spec}(R))$ is annihilated by an integer $D > 0$, written as $D \cdot \operatorname{Cl}(R) = 0$, the symbolic power $P^{(D(r-1)+1)} \subseteq P^r$ for all $r> 0$  and all prime ideals $P \subseteq R$ of height one. 
\end{lem}

Working over an algebraically closed field $\F$, fix a pointed cone $C$ as in Remark \ref{rem:strong-convexity-stipulation} and the pair of rings $R_\F$ and $R'_\F$ as in Proposition \ref{prop: proof reduction}. 
When $C \neq \{0\}$, each $\rho \in \Sigma(1)$, the collection of \textbf{rational rays (one-dimensional faces)} of $C$,  yields a unique primitive generator $u_\rho \in \rho \cap N$ for $C$ and a torus-invariant height one prime ideal $P_\rho$ in $R'_\F$; cf.,  \cite[Thm.~3.2.6]{torictome}.  
The torus-invariant height one primes generate a free abelian group $\bigoplus_{\rho \in \Sigma(1)} \Z  P_\rho$ which maps surjectively onto the divisor class group of $R'_\F$.  More precisely, we record the following well-known theorem; see \cite[Ch.~4]{torictome}.\footnote{This result follows from \cite[Prop.~3.3.9, Prop.~4.1.1-4.1.2, Thm.~4.1.3, Exer.~4.1.1-4.1.2, Prop.~4.2.2, Prop.~4.2.6, and Prop.~4.2.7]{torictome}, essentially consolidating what facts we need to bear in mind going forward in the manuscript.} 
\begin{theorem}\label{thm: exact sequence}
\textit{
With notation as in Proposition \ref{prop: proof reduction}, let $C \subseteq N_\R$ be a pointed cone with primitive  generators $\Sigma(1)$ as described above. %Let $C' = C \cap N'_\R$ be the corresponding pointed full-dimensional cone. 
Then there is a short exact sequence 
\begin{equation}\label{eqn:SES001}
0 \to M' \stackrel{\phi}{\to} \bigoplus_{\rho \in \Sigma(1)} \Z  P_\rho \to \operatorname{Cl}(R'_\F) \to 0,
\end{equation} 
where $\phi (m) = \operatorname{div}(\chi^m)=  \sum_{\rho \in \Sigma(1)}  \langle m , u_\rho\rangle  P_\rho$. %, and $\langle \cdot, \cdot \rangle$ is our bilinear pairing. 
Furthermore, $\operatorname{Cl}(R_\F)$ and $\operatorname{Cl}(R'_\F)$ are isomorphic,  $\operatorname{Cl}(R_\F)$ is finite abelian if and only if $C$ is simplicial, and trivial if and only if  $C$ is smooth.
}
\end{theorem}
\noindent By definition, the cone $C  \subseteq N_\R$ is \textbf{simplicial (respectively, smooth)} if $C = \{0\}$ or the primitive ray generators form part of  an $\R$-basis for $N_\R$ (resp., a $\Z$-basis for $N$). We also apply the adjectives simplicial and smooth to the corresponding toric algebra $R_\F$ and the toric $\F$-variety $\operatorname{Spec}(R_\F)$.

\section{Main Results}\label{sec:main-results}

Maintaining all notation conventions from the last section, we now state our main results.

\begin{theorem}\label{thm:ITAU-sandwhiching000}
\textit{
Let $C \subseteq N_\R$ be a full pointed rational polyhedral cone. Let $R_\F = \F [C^\vee \cap M]$ be the associated toric algebra over a field $\F$. Set $D := \max_{m \in \mathcal{B}
} \langle m, v_C \rangle,$ where $\mathcal{B}$ is the minimal generating set for $C^\vee \cap M$ and $v_C \in N$ is the sum of the primitive generators for $C$. Then $$P^{(D(r-1)+1)} \subseteq P^r$$ for all $r>0$, and all monomial primes $P$ in $R_\F$.}  
\end{theorem}

\begin{theorem}\label{thm:height-one-primes}
\textit{
With notation as in Theorem \ref{thm:ITAU-sandwhiching000}, we assume further that $C$ is simplicial. Define 
$T := \max  \left\lbrace  \max_{m \in \mathcal{B}} \langle m, v_C \rangle, D \right\rbrace,$ where $D$ is any positive integer such that $D \cdot \operatorname{Cl}(R_\F) = 0$.  
Then 
$$P^{(T(r-1)+1)} \subseteq P^r$$ for all $r>0$, all monomial primes, and all height one primes in $R_\F$. In particular, we can take $D = \# \operatorname{Cl}(R_\F) $. }
\end{theorem}

\begin{remark}\label{rem:smooth-cones} 
If the cone $C$ in Theorem \ref{thm:height-one-primes} is smooth, then $T = 1$ and $P^{(r)} = P^r$ for all $r>0$, all monomial primes, and all height one primes in $R_\F$.  
As $C$ is smooth, $C$ and $C^\vee$ are generated by a $\Z$-basis for $N$ and the dual basis for $M$, respectively.  Also, $\# \operatorname{Cl} (R_\F) = \# \operatorname{Cl}  (R_{\overline{\F}}) = 1$.  
Note that in general, this means our multiplier $T$ will not confirm uniform symbolic topologies for all primes $P$ in a toric algebra. For example, even in a polynomial ring of dimension three, there are height two primes for which $P^{(r)} \neq P^r$ for some $r \ge 2$; \cite[p.2 of Intro]{5authorSymbolicSurvey} gives an example. 
\end{remark}

\begin{remark}\label{rem:surface-case}
Two-dimensional toric algebras are always simplicial with cyclic class class group. In this case, the conclusion of  
Theorem \ref{thm:height-one-primes} holds using the multiplier $\# \operatorname{Cl}  (R_{\F})$. 
This multiplier is sharp by Proposition \ref{prop:sharpness-in-height-one}. 
\end{remark}

\begin{rem}\label{rem:neither-full-nor-pointed} 
Theorems \ref{thm:ITAU-sandwhiching000}-\ref{thm:height-one-primes} can be adapted to the non-full case by replacing $R_\F$ with $R_\F'$ as in Proposition \ref{prop: proof reduction}, and  applying \cite[Prop.~2.6]{Walker003} to the faithfully flat map $\varphi$ from Proposition \ref{prop: proof reduction}.  
\end{rem}

\begin{proof}[Proof of Theorem \ref{thm:ITAU-sandwhiching000}] 
We may fix a face $F \neq \{0\}$ of $C$, and $P = P_F$ the corresponding  monomial prime in $R = R_\F$. 
Per Lemma \ref{lem:prime-positivity} \eqref{eqn:mono-prime-defn},  
a monomial $\chi^m \in P = P_F$ if and only if $\left\langle m,  v_F \right\rangle \in  \Z_{>0},$ where $v_F \in F \cap N$ is the sum of the primitive generators for $F$.   
\begin{lemma}\label{lem:monomial-sandwiching01} \textit{For each integer $E \ge 1$,  $P_F^{(E)} \subseteq I_F (E)$, where $I_F (E) := \left( \chi^{m} \colon   \left\langle m , v_F  \right\rangle  \ge E \right) R.$}  
\end{lemma}
\noindent First, $I_F(E)$ is $P_F$-primary for all $E \ge 1$, 
 i.e., if $s f \in I_F (E)$ for some $s \in R - P_F$, then $f \in I_F (E)$. As $I_F(E)$ is monomial, we may test this by fixing $\chi^m \in I_F (E) R_{P_F} \cap R$ and $\chi^q \in R-P_F$ such that $\chi^m \cdot \chi^q = \chi^{m+q} \in I_F (E)$: $\left\langle q, v_F  \right\rangle = 0, \mbox{ while }  E \le \left\langle m + q, v_F  \right\rangle = \left\langle m , v_F  \right\rangle + \left\langle q, v_F  \right\rangle = \left\langle m, v_F \right\rangle,$ so $\chi^m \in I_F (E)$. 
  Thus all $I_F(E)$ are $P_F$-primary, and certainly $P_F^E \subseteq I_F (E)$. Thus $P_F^{(E)} \subseteq I_F (E)$, being the smallest $P_F$-primary ideal containing $P_F^E$, proving the lemma. 

Certainly, $1 \le \max_{m \in \mathcal{B}} \left\langle m, v_F \right\rangle \le  \max_{m \in \mathcal{B}} \left\langle m, v_C \right\rangle$, 
for $\mathcal{B}$ and $D$ as above. 
The conclusion of the theorem follows by \cite[Lem.~3.3]{Walker002}, once we verify the left-hand containments in the next  

\begin{lemma}\label{lem:monomial-sandwiching02} \textit{For each integer $E \ge 1$, $I_F (E) \subseteq P_F^{\lceil E/D' \rceil} \subseteq P_F^{\lceil E/ D \rceil }$ where $D' = \max_{m \in \mathcal{B}} \left\langle m, v_F \right\rangle.$}  
\end{lemma}
\noindent  Fix any monomial $\chi^\ell \in I_F (E)$, say $\ell = \sum_{m \in \mathcal{B}} a_m \cdot   m$ with $a_m \in \Z_{\ge 0}$. Let $S \subseteq \mathcal
B$ consist of those $m \in \mathcal{B}$ such that the monomials $\chi^m$ form a minimal generating set for $P$.   
By linearity of $\langle \bullet , v_F \rangle$,  
$$E \le  \left\langle \ell , v_F  \right\rangle = \sum_{m \in  \mathcal{B} } a_m \left\langle m, v_F  \right\rangle = \sum_{m \in  S}  a_m \left\langle m, v_F  \right\rangle  \le \sum_{m \in S} a_m \cdot D'  \Longrightarrow \sum_{m \in S} a_m \ge \lceil E / D' \rceil. $$ 
Thus $\chi^m \in P_F^{\sum_{m \in S } a_m } \subseteq P_F^{\lceil E / D' \rceil}$. 
Being a monomial ideal, it follows that $I_F (E) \subseteq P_F^{\lceil E/D' \rceil}$. 
\end{proof}

\begin{remark} In passing, we invite the reader to compare the ideals $I_F (\bullet)$ in Lemma \ref{lem:monomial-sandwiching01} with Bruns and Gubeladze's terminology and description \cite[Ch.~4, p.~149]{BG-Polytopes-Ktheory} for the symbolic powers of the height one monomials primes in terms of  a full pointed cone. 
Lemma \ref{lem:monomial-sandwiching01} works in any height.  
\end{remark}

\begin{lemma}\label{lem:base-change-irrelevant}
\textit{With notation as in Theorem \ref{thm:ITAU-sandwhiching000}, the class groups %\ref{thm:height-one-primes}, 
$\operatorname{Cl}(R_\F) \cong \operatorname{Cl}(R_{\overline{\F}})$ are isomorphic.}
\end{lemma}

\begin{proof} Since $C$ is a full pointed cone in Theorem \ref{thm:ITAU-sandwhiching000}, $R_\F$ admits an $\N$-grading; see the passage above Remark \ref{rem:strong-convexity-stipulation}. We may then cite Fossum \cite[Cor.~10.5 on p.43]{fossum} to conclude that up to isomorphism, $\operatorname{Cl}(R_\F) \subseteq \operatorname{Cl}(R_{\overline{\F}})$ as a subgroup. This is an equality for toric rings because the divisor classes of height one monomial primes belong to both groups and generate the latter by Theorem \ref{thm: exact sequence}. 
\end{proof}

\begin{proof}[Proof of Theorem \ref{thm:height-one-primes}] Since $C$ is simplicial, $\# \operatorname{Cl}(R_\F)$ is finite by Lemma \ref{lem:base-change-irrelevant} and 
Theorem \ref{thm: exact sequence}. Now we simply combine Theorem \ref{thm:ITAU-sandwhiching000} with Lemma \ref{thm: du Val bound 1}, and take the maximum of the values.  
\end{proof}

\begin{exm}\label{exm:Hypersurf-000}
Fix an arbitrary ground field $\F$ and integers $n \ge 2$ and $E \ge 2$. Let $$R = \frac{\F [x_1, \ldots, x_n , z]}{(z^E -  x_1 \cdots x_n)}.$$ Then $P^{(T (r-1) + 1)} \subseteq P^r$ for all $r>0$, all monomials primes, and all height one primes in $R$, where $T = \max \{n, E \}$. Indeed, $R$ is a toric algebra arising from the simplical full pointed cone $C \subseteq \RR^n$ spanned by $$\{e_n, \mbox{ } E \cdot e_i + e_n \colon i = 1, \ldots, n-1\} \subseteq \Z^n,$$ where $e_1, \ldots, e_n$ denote the standard basis vectors in $\RR^n$. We may apply Theorem \ref{thm:height-one-primes}, noting that: \begin{itemize}
\item In the notation of Theorem \ref{thm:ITAU-sandwhiching000}, $\mathcal{B} = \{e_1,  \ldots, e_{n-1}, e_n,  E \cdot e_n - e_1 -  \cdots -  e_{n-1}\} \subseteq \Z^n$ and the vector $v_C = n \cdot e_n + E \cdot (e_1 + \cdots + e_{n-1}) \in \Z^n$. 
\item $\operatorname{Cl}(R)  \cong (\Z  / E  \Z)^{n-1},$ so $E \cdot \operatorname{Cl}(R) = 0$; see  \cite[Sec.~4]{Walker002} for this class group  computation. 
\end{itemize}
\end{exm}

\subsubsection{Von Korff's Toric F-Signature Formula}
Now fix a perfect field $\KK$ of positive characteristic $p$. Given an F-finite $\N$-graded domain $R$ of finite type over $\KK$, for each integer $e \ge 0$, we have an $R$-module isomorphism %of $R$-modules 
$R^{1/p^e} \cong R^{a_e} \oplus M$ where $M$ has no free summand, and the integer $a_e \le p^{e d}$ where $d = \dim R$. By definition, the \textbf{F-signature} of $R$ is (see \cite{HL02} and \cite{KT11}) 
$$s(R) := \limsup_{e \to \infty} \frac{a_e}{p^{ed}}  =  \lim_{e \to \infty} \frac{a_e}{p^{ed}},  \quad 0 \le s(R) \le 1. $$ %it was defined formally in \cite{HL02} first as (1), while (2) was confirmed in full generality by Tucker \cite[Thm.~4.9]{KT11}.  
 The F-signature has ties to measuring F-singularities: for instance, $s(R)$ is positive if and only if $R$ is strongly F-regular \cite{AL03}, and $s(R) = 1$ if and only if $R$ is regular \cite[Thm.~4.16]{KT11}. 

Over the perfect field $\KK$, 
any normal toric ring is strongly F-regular and its F-signature is rational \cite{Sin05}. 
We now state Von Korff's result 
\cite[Thm.~3.2.3]{MVKorff}; see also Watanabe-Yoshida \cite[Thm.~5.1]{WY04} and 
Yao \cite[Rem.~2.3(4)]{Yao06}:  

\begin{theorem}[cf., Von Korff  {\cite[Thm.~3.2.3]{MVKorff}}]\label{thm: WYVK Poly FSig}
\textit{
With notation as in Proposition \ref{prop: proof reduction}, we define a convex polytope, 
$P_{C'} := \{w \in M'_\R \colon 0 \le  \langle w, v \rangle < 1 , \forall v \in G\} \subsetneqq (C')^\vee,$ 
where $G$ is the set of primitive generators of $C' \neq \{0\}$. Then over any perfect field $\KK$ of positive characteristic,  
the F-signature $s(R_\KK) = s(R'_\KK) = \operatorname{Vol}(P_{C'}) \in \Q_{>0}$, where the volume form $\operatorname{Vol}$ on $M'_\RR$ is chosen so that a hypercube spanned by primitive generators of $M'$ has volume one.}  \end{theorem}  

\begin{remark}\label{exm:simplicial-Fsig} 
When the cone $C'$ in Theorem \ref{thm: WYVK Poly FSig} is simplicial, $s(R_\KK') = 1 /D$ where the integer $D = \# \operatorname{Cl}(R_{\overline{\KK}}') = \# \operatorname{Cl}(R_{\KK}')$; cf.,   \cite[Cor.~3.2]{Walker001}. The latter equality holds by Lemma \ref{lem:base-change-irrelevant}. %Specifically, $P_{C'}$ is the preimage image of the unit hypercube under a linear isomorphism $L \colon M'_\RR \to \RR^a$ with $\det(L) = \# \operatorname{Cl}(R_{\overline{\KK}})$, so $\operatorname{Vol}(P_{C'})  =  1 / \det (L)$. %,   extending scalars to the algebraic closure of $\F$. 
%If $\dim (C) = \dim (N_\R) = n$, let $u_1, \ldots, u_n$ be the primitive generators for $C$. For convenience of presentation, we implicitly fix a $\Z$-basis for $N$ with induced dual $\Z$-basis for $M$, and use $n$-tuple coordinate notation so that the eventual tie-in with calculus is easily rendered. We index the $u_i = (a_{i, 1}, \dots, a_{i, n})$ so that they form the rows of a positive-definite integer matrix $A$. Defining new  coordinates $\ell_i := \ell_i (m) = \langle m, u_i \rangle$, for all $m = (x_1, \ldots, x_n) \in M_\R$, we set $$\mbox{\mancube}_n := \{(\ell_1, \ldots, \ell_n) \in M_\R \colon 0 \le \ell_i < 1\}$$ 
 %the unit $n$-cube in $\RR^n$ with
% in the coordinates $\ell_1, \ldots, \ell_n$ defined above, let
%Let $L_A \colon (M_\R)_{(x_1, \ldots, x_n)} \stackrel{\cong}{\to} (M_\R)_{(\ell_1, \ldots, \ell_n) }$ 
%denote the $\RR$-linear change-of-coordinates map defined by $A$, and $J = | \det \left(\frac{\partial \ell_i (x_1, \ldots, x_n)}{\partial x_j} \right) | = \det (A)$; $P_C = L_A^{-1} (\mbox{\mancube}_n)$, so the change-of-variables formula \textbf{(!)} from calculus implies the third equality below,  where $\int^n$ stands for taking an $n$-tuple integral:
%\begin{align*} 1 = \operatorname{Vol}(\mbox{\mancube}_n) = \int_{\mbox{\mancube}_n}^n 1 d \ell_1 d \ell_2 \cdots d \ell_n \stackrel{(!)}{=} \int_{P_C}^n J  d x_1 d x_2 \cdots d x_n  = \det(A) \cdot \operatorname{Vol}(P_C).  \end{align*}  So $s(R_\F) = \operatorname{Vol}(P_C) = 1/ \det(A)$, and $1 /s(R_\F) = \det (A) = \#\mbox{Cl}(R_{\overline{\F}})$, per  \cite[Cor.~3.2]{Walker001}.  
\end{remark} 

\noindent Given Theorem \ref{thm: WYVK Poly FSig}, 
 we could opt to replace the invariant $T$ from Theorem \ref{thm:height-one-primes} with the possibly larger invariant $U = \operatorname{lcm} \{ \max_{m \in \mathcal{B} } \langle m , v_C \rangle , \# \operatorname{Cl} (R_\KK) \} \in \Z_{>0}$, the least common multiple considered in Theorem \ref{thm:VarPi-FSig-Bound01}. 
 The latter is an integer multiple of the reciprocal $1 / s(R_\KK)= \#\operatorname{Cl}(R_\KK)$ of the F-signature of $R_\KK$ that controls the asymptotic growth of symbolic powers.

\section{Examples of (Non-)Sharp Multipliers; Segre-Veronese algebras}\label{sec:Examples} 

To start, we deduce a result that occasionally provides sharp multipliers in the toric setting. 
\begin{proposition}\label{prop:sharpness-in-height-one}
\textit{With notation as in Theorem \ref{thm:height-one-primes}, we assume $C$ is a simplicial full pointed rational polyhedral cone. We now set $B  := \max_{w \in \mathcal{PG}} \left\langle w, v_C \right\rangle$ where $\mathcal{PG} \subseteq \mathcal{B}$ consists of the primitive generators of $C^\vee$. There exists a monomial prime $P$ in $R = R_\F$ of height one
such that:
\begin{enumerate}
\item $P^{(B(r-1))} \not\subseteq P^r$ for some $r \ge 2$; 
\item There is no positive integer $D' < B$ such that $P^{(D'(r-1)+1)} \subseteq P^r$ for all $r > 0$.
\end{enumerate}
}
\end{proposition}

\begin{proof} 
Let $v_1, \ldots, v_n \in N$ and $w_1, \ldots, w_n \in M$ denote the primitive generators for $C$ and for $C^\vee$, respectively. We index these generators so that the nonnegative integer $\langle w_j, v_i \rangle$ is positive if and only if $i=j$: we may do this citing  the notion of facet normals 
 \cite[after Prop.~1.2.8]{torictome}. In deference to Lemma \ref{lem:prime-positivity}\eqref{eqn:mono-prime-defn}, let $P_j$ ($1 \le j \le n$) be the height one monomial prime in $R_\F$ such that a monomial $\chi^m  \in P_j$ if and only if $\langle m, v_j \rangle >  0$. In particular, $\chi^{w_j} \in P_j$ for each $j$. 
\begin{lem*}\label{thm:VarPi-FSig-Bound00}
\textit{For each $1 \le j \le n$, $\langle w_j , v_j \rangle $ is the order of the element in $\operatorname{Cl}(R_{\F})$ corresponding to $P_j$.  }
\end{lem*}
\noindent We may leverage exact sequence \eqref{eqn:SES001} from Theorem \ref{thm: exact sequence}, since Lemma \ref{lem:base-change-irrelevant} allows us reduce to the case where $\F$ is algebraically closed.  
For $1 \le j \le n$, we have $0 = [\operatorname{div}(\chi^{w_j})] =  \langle w_j, v_j \rangle [D_{\rho_j}]$, where $\rho_j$ is the rational ray of $C$ generated by $v_j$. 
Thus $P_j^{ (\langle w_j, v_j \rangle) } = (\chi^{w_j})R.$ 
Since the order of $[D_{\rho_j}]$ is the smallest $E_j>0$ such that $P_j^{(E_j)} = (\chi^{m_j})R$ is principal for some $m_j \in C^\vee \cap M -\{0\}$, $E_j$ divides $\langle w_j, v_j \rangle$, and $P_j^{(\langle w_j, v_j \rangle)} = (P_j^{(E_j)})^L$ where $L = \langle w_j, v_j \rangle / E_j$. 
%Two principal ideals coincide if and only if the chosen generators are the same up to rescaling by a unit. The only unit monomial is $\chi^0 =1$ since 
Since $C^\vee$ is strongly convex, we may conclude that  $\chi^{w_j} = \chi^{L \cdot m_j}$, and $L=1$  since $w_j$ is an irreducible vector in $C^\vee \cap M$; see Subsection \ref{subsub:Hilb00}. This proves the lemma. 

\noindent To prove (1), 
notice $B = \langle w_{j_0}, v_{j_0} \rangle$ for some $1 \le j_0 \le n$. 
Then (*): $(\chi^{w_{j_0}})R = P_{j_0}^{(B)} \not\subseteq P_{j_0}^2$. We recall from Section \ref{sec:Toric-Alg-Prelims00} that when the cone $C$ is full-dimensional in $N_\R$, 
the semigroup algebra $R_\F = \F [C^\vee \cap M]$ can be $\N$-graded. This means that any minimal generator $f$ of a homogeneous ideal $I$ satisfies $f \in I - I^2$ by Nakayama's lemma. In our situation, $I = P_{j_0}$ and $f = \chi^{w_{j_0}}$. 
Observation (*) also gives part (2), arguing by contradiction and using \cite[Lem.~3.3]{Walker002} accordingly.
\end{proof}

We offer several examples to show that establishing sharpness of our bilinear multipliers %on monomial primes of height two or more 
is a delicate  matter meriting further study. 

\begin{exm}\label{exm:Vero-Hypersurf00}
We fix integers $n \ge 2$ and $E \ge 2$, and an arbitrary field $\F$. Let $V_{E, n}$ be the $E$-th Veronese subalgebra of the polynomial ring $\F [x_1, \ldots, x_n]$, that is, the $\F$-algebra generated by all monomials of degree $E$ in $x_1, \ldots, x_n$. Then $P^{(E (r-1) + 1)} \subseteq P^r$ for all $r>0$, all monomial primes, and all primes of height one by Theorem \ref{thm:height-one-primes}. 
See \cite[Sec.~4]{Walker002} for details. However, for any $E' < E$, the proof of Proposition \ref{prop:sharpness-in-height-one} guarantees that we can find a prime $P \subseteq V_{E, n}$ (monomial, height one) such that $P^{(E' (r-1)+1 )} \not\subseteq P^r$ for some $r \ge 2$, namely, for $r = 2$. In fact, this last observation holds for all monomial primes in $V_{E,n}$, aside from the zero ideal and the maximal monomial ideal for which $E'=1$ will do; the proof of \cite[Thm.~4.3]{Walker002} confirms this explicitly.    
\end{exm}

Despite Example \ref{exm:Vero-Hypersurf00}, Theorem \ref{thm:height-one-primes} does not give sharp multipliers in general. For example,  

\begin{exm}
For any $n > 2$, let $R = \F [Z, X_1, \ldots, X_n] / (Z^2 - X_1 \cdots X_n)$ as in  Example \ref{exm:Hypersurf-000}.  Citing the proof of \cite[Thm.~4.1]{Walker002}, when $P \subseteq R$ is any monomial prime of height at least 2:
\begin{itemize}
\item $P^{(r)} = P^r$ for all $r >0$; however,    
\item The multiplier $D'$ corresponding to $P$ in 
Lemma \ref{lem:monomial-sandwiching02} always satisfies  $D' \ge 2$. 
\end{itemize}
\end{exm}

Theorem \ref{thm:ITAU-sandwhiching000} gives a uniform multiplier $D$ that works for all monomial primes. Even when this multiplier is sharp across all monomial primes, it need not be best possible for all monomial primes of a given height, contrasting with the situation of Example \ref{exm:Vero-Hypersurf00}. For example,   

\begin{exm}\label{exm:poly-Segre-prod00}
Let $R = \F [C^\vee \cap \Z^3] = \F [x, y, z, w] / (xy- zw)$, for the non-simplicial cone $C \subseteq \R^3$ with  $e_1, e_2, e_1 + e_3 , e_2+e_3 $ as primitive generators. Theorem \ref{thm:ITAU-sandwhiching000} says $P^{(2 r - 1)} \subseteq P^r$ for all $r>0$ and all monomial primes in $R$, observing that $C^\vee \cap \Z^3$ is minimally generated by $\mathcal{B} = \{e_1, e_2, e_3, (1, 1, -1)\}$. Given any height two monomial prime $P$ in $R$, these containments cannot be improved to $P^{(r)} = P^r$ for all $r \ge 2$. For instance, if $P = (x, y, z)R$, then for any $s \ge 1$, $z^s \in P^{(2s)} -  P^{2s}$ and $z^{s+1} \in P^{(2 s + 1)} - P^{2s+1}$: indeed, $w^s \in R- P$ and $R$ can be standard graded, so the least degree of a homogeneous element of $P^r$ is $r$. By contrast, $P^{(r)} = P^r$ for all $r$  and for any height one monomial prime $P$ in $R$: the invariant $D' = 1$ in Lemma \ref{lem:monomial-sandwiching02} via direct computation. 
\end{exm}

\subsubsection{Final Example Computation: Segre-Veronese algebras} In what follows,  $\F$ is a fixed arbitrary field. For more on Segre products, \cite{GotoWatanabeGraded1} is often cited as a standard reference.

\begin{definition}\label{defn:Segre-finite-type}
Fix a family $A_1, \ldots, A_k$ of $k$ standard graded algebras of finite type over $\F$, with $A_i = \F [a_{i, 1}, \ldots, a_{i, b_i}]$ in terms of algebra generators. Their \textbf{Segre product} over $\F$ is the ring $S = (\#_\F )_{i=1}^k  A_i$ generated up to isomorphism as an $\F$-algebra by all $k$-fold products of the $a_{i,j}$. 
\end{definition}

\begin{definition}\label{defn:Vero}
We fix integers $E \ge 1$ and $m \ge 2$. Suppose $A = \F [x_1, \ldots, x_m]$ is a standard graded polynomial ring in $m$ variables over a field $\F$. 
Let $V_{E, m} \subseteq A$ denote the $E$-th \textbf{Veronese subring} of $A$, the standard graded $\F$-subalgebra generated by all monomials of degree $E$ in the $x_i$. There are $\binom{m-1 + E}{E}$ such monomials; this number is the \textbf{embedding dimension} of $V_{E, m}$.
\end{definition}

\begin{definition}\label{defn:Segre-Vero}
Fix $k$-tuples $\overline{E} = (E_1, \ldots, E_k) \in (\Z_{\ge 1})^k$ and  $\overline{m} = (m_1, \ldots, m_k) \in (\Z_{\ge 2})^k$ of integers, with $k \ge 1$.  Furthermore, we set $d(j) = \left( \sum_{i=1}^j m_i \right) - (j-1) $ for each $1 \le j  \le k$: $d(k)$ is the Krull dimension of the Segre product  $SV \left(\overline{E},\overline{m}\right) = (\#_\F)_{i=1}^k V_{E_i, m_i}$ of $k$ Veronese rings  in $m_1, \ldots, m_k$ variables, respectively; this is  a \textbf{Segre-Veronese algebra with degree sequence $\overline{E}$.}
\end{definition}

\noindent Over any perfect field $\KK$, a Segre-Veronese algebra has  uniform symbolic topologies on all primes, per \cite[Thm.~2.2]{ELS}, \cite[Thm.~1.1]{HH1}, and \cite[Cor.~3.10]{HKV}. However, explicit $D$ are elusive unless $k=1$; see \cite[Cor.~3.30]{5authorSymbolicSurvey}. We provide an effective $D$ for the monomial primes, which is sharp as suggested by Examples \ref{exm:Vero-Hypersurf00} and \ref{exm:poly-Segre-prod00}, the latter being the Segre product of  polynomial rings in two variables:

\begin{theorem}\label{thm:SegreVeroMonoUSTP01}
\textit{
Suppose $A = SV \left(\overline{E},\overline{m}\right)$ is a Segre-Veronese algebra over $\F$ with degree sequence $\overline{E} = (E_1, \ldots, E_k)$. Let $D := \sum_{i=1}^k E_i$. Then $P^{(D (r-1)+1)} \subseteq P^r$ for all $r>0$ and all monomial primes $P$ in $A$.  
}
\end{theorem}

\begin{proof} 
Given a lattice $N \cong \Z^d$ we will use $e_1, \ldots, e_d \in N$ to denote a choice of basis for $N$ will dual basis $e_1^* , \ldots, e_d^*$ for $M$.  
In the setup of Theorem \ref{thm:ITAU-sandwhiching000}, the cardinality of the minimal generating set $\mathcal B$ of $C^\vee \cap M$ is the \textbf{embedding dimension} of the toric algebra  $R_\F = \F [C^\vee \cap M]$; we refer the reader to \cite[Sec.~1.0, Proof of Thm.~1.3.10]{torictome}. 

\begin{lem*}\label{lem:Segre-Vero-ToricPresentation01}
\textit{
For $A$ as stated, there exists a lattice $N$ and a full pointed rational polyhedral cone $C \subseteq N_\RR$ such that $A = \F [C^\vee \cap M]$ up to isomorphism. 
}
\end{lem*}

\begin{proof}
Fix $k$-tuples $\overline{E} \in (\Z_{\ge 1})^k$ and  $\overline{m} \in (\Z_{\ge 2})^k$. Set $d(j) = \left( \sum_{i=1}^j m_i \right) - (j-1)$ for $1 \le j  \le k$, while $d (0) = 0$. 
Given $SV \left(\overline{E},\overline{m}\right) = (\#_\F)_{i=1}^k V_{E_i, m_i}$, we fix a lattice $N \cong \Z^{d(k)}$ and record a cone $C = C \left(\overline{E},\overline{m}\right) \subseteq N_\RR \cong \RR^{d(k)}$ as stipulated with $R_\F = \F [C^\vee \cap M] \cong SV \left(\overline{E},\overline{m}\right)$. Specifically, consider the %full-dimensional strongly convex rational 
cone $C \subseteq N_\RR$ generated by the following irredundant collection of primitive vectors:  
\begin{align*}
\mathcal{A} = \bigcup_{1 \le j \le k} A_j, \mbox{ where } A_1 &= \{e_1, \ldots, e_{m_1 - 1},  - e_1 -  \cdots - e_{m_1 - 1} + E_1 \cdot e_{m_1}\}, \\ 
\mbox{ and for each }2 \le j \le k, \quad A_j &=  \left\lbrace e_h, \mbox{ } E_j \cdot e_{m_1} - \sum_{h = d(j-1)+1}^{d(j)} e_{h} \colon d(j-1) + 1 \le h  \le d(j) \right\rbrace.  
\end{align*} 
The semigroup $C^\vee \cap M$ is generated by  the following set of irreducible vectors:
\begin{align*}
\mathcal{B} = \left\lbrace e^*_{m_1} +  \sum_{j=1}^k \sum_{\ell = 1}^{m_j - 1} a_{j , \ell} \cdot e^*_{d(j-1) + \ell}  \colon  0 \le \sum_{\ell =1}^{m_j - 1} a_{j, \ell} \le E_j \mbox{ for }1 \le j \le k \right\rbrace.
\end{align*}
Indeed, $\# \mathcal{B} = \prod_{j=1}^k \binom{m_j - 1 + E_j}{E_j}$, the embedding dimension of $SV (\overline{E} , \overline{m})$.  
Finally, one can record a bijection between the monomial generators of $R_\F$ and those typically used to present $SV (\overline{E} , \overline{m})$; cf., \cite[Proof of Lem.~4.2]{Walker002} for how the bijection would look in the coordinates $a_{j, \ell}$ for each $j$. 
\end{proof}
\noindent 
Feeding $v_C = \sum_{u \in \mathcal{A}} u = (\sum_{j=1}^k E_j) \cdot e_{m_1}$ and $\mathcal{B}$ into Theorem \ref{thm:ITAU-sandwhiching000}  yields $D = \sum_{j=1}^k E_j$.  We win! 
\end{proof}

In passing, we note that the corresponding class of varieties (Segre-Veronese varieties) form a cornerstone of a lot of investigations in classical and applied algebraic geometry, often being varieties for which one has a more incisive handle on some specified computational- or other task. 
As a sampling of recent investigations in this vein, we include \cite{BallicoBernardi1701,BallicoBernardi1706,BallicoBernardi1707,DufresneJeffries1602,Schreyer1707} among our references.


\begin{thebibliography}{1}

\bibitem{AL03}  I.M. Aberbach and G.J. Leuschke. {\em The F-signature and strong F-regularity}. Math. Res. Lett. \textbf{10} (2003), no. 1, pp. 51--56.

%\bibitem{Akes01} S. Akesseh.  {\em Ideal Containments Under Flat Extensions.}  \href{http://arxiv.org/abs/1512.08053}{arXiv preprint} 

%\bibitem{Alt-Klei13} A. Altman and S. Kleiman. {\em A Term of Commutative Algebra}. Worldwide Center of Mathematics LLC, Cambridge, MA, 2014.   
       
%\bibitem{AtiyahMacdonald69} M. Atiyah, I.G. Macdonald. {\em Introduction to Commutative Algebra.} Addison-Wesley, Reading, MA, 1969.   

\bibitem{BallicoBernardi1701} E. Ballico, A. Bernardi. {\em On the ranks of the third secant variety of Segre-Veronese embeddings.} \href{https://arxiv.org/abs/1701.06845}{arXiv/1701.06845} 

\bibitem{BallicoBernardi1706} E. Ballico, A. Bernardi, L. Chiantini. {\em On the dimension of contact loci and the identifiability of tensors.} \href{https://arxiv.org/abs/1706.02746}{arXiv/1706.02746} 

\bibitem{BallicoBernardi1707} E. Ballico, A. Bernardi, F. Gesmundo. {\em A note on the cactus rank for Segre-Veronese varieties.} \href{https://arxiv.org/abs/1707.06389}{arXiv/1707.06389} 


%\bibitem{BHPV00} W.P. Barth, K. Hulek, C.A.M. Peters, A. Van de Ven, {\em Compact Complex Surfaces,} second edition. Springer-Verlag Berlin, 2004.

\bibitem{Primer} T. Bauer, S. Di Rocco, B. Harbourne, M. Kapustka, A.L. Knutsen, W. Syzdek, T. Szemberg. {\em A primer on Seshadri constants.} Contemporary Mathematics \textbf{496}  (2009), pp. 33-70. 
  

\bibitem{BG-Polytopes-Ktheory} W. Bruns, and J. Gubeladze. {\em Polytopes, rings, and K-theory.} Springer Science \& Business Media, 2009. 

  
%  \bibitem{Claborn} L. Claborn. {\em Every abelian group is a class group.} Pacific J. Math. \textbf{18}, Number 2 (1966), pp. 219-222.

%  \bibitem{resurge0} C. Bocci, S. Cooper, and B. Harbourne. {\em Containment results for ideals of various configurations of points in $\mathbb{P}^n$.} J. Pure Appl. Algebra \textbf{218} (2014), no.1, pp. 65-75.
  
 %  \bibitem{BH1} C. Bocci and B. Harbourne. {\em Comparing powers and symbolic powers of ideals.} J. Algebraic Geom. \textbf{19} (2010), no.3, pp. 399-417.

%\bibitem{BCGH-MFO} C. Bocci, E. Carlini, E. Guardo, B. Harbourne. {\em Mini-Workshop: Ideals of Linear Subspaces, Their Symbolic Powers and Waring Problems.} Oberwolfach Rep. \textbf{12} (2015), 489-532. doi: 10.4171/OWR/2015/9 \href{http://www.ems-ph.org/journals/show_abstract.php?issn=1660-8933&vol=12&iss=1&rank=9}{Online Here}

%\bibitem{BCGHJNSVTV-000} C. Bocci, S. Cooper, E. Guardo, B. Harbourne, M. Janssen, U. Nagel, A. Seceleanu, A. Van Tuyl, T.  Vu. {\em The Waldschmidt constant for squarefree monomial ideals.} To appear in the J. Algebraic Combin.   \href{http://arxiv.org/abs/1508.00477}{arXiv preprint}

%\bibitem{CFGLMNSSV00} S. Cooper, G. Fatabbi, E. Guardo, A. Lorenzini, J. Migliore, U. Nagel, A. Seceleanu, J. Szpond, A. Van Tuyl.  {\em Symbolic powers of codimension two Cohen-Macaulay ideals.} \href{http://arxiv.org/abs/1606.00935}{arXiv preprint}

%\bibitem{CooperHartke00} S.M. Cooper, S.G. Hartke {\em The Alpha Problem \& Line Count Configurations.} \href{http://arxiv.org/abs/1312.4147}{arXiv preprint}

  \bibitem{torictome} D.A. Cox, J.B Little, and H.K. Schenck. {\em Toric Varieties} (2011), Graduate Studies in Mathematics 124.\\ American Mathematical Society, Providence, RI.

  %\bibitem{resurge1} M. Dumnicki, B. Harbourne, U. Nagel, A. Seceleanu, T. Szemberg, and H. Tutaj-Gasi\'{n}ska. {\em Resurgences for ideals of special point configurations in $\mathbb{P}^n$ coming from Hyperplane Arrangements.} (2014) arXiv:1404.4957v1. 

\bibitem{5authorSymbolicSurvey} H. Dao, A. De Stefani, E. Grifo, C. Huneke, L. Nu\~{n}ez-Betancourt. {\em Symbolic Powers of Ideals.} \href{https://arxiv.org/abs/1708.03010}{arXiv/1708.0301} 


\bibitem{DufresneJeffries1602} E. Dufresne, J. Jeffries. {\em Mapping toric varieties into low dimensional spaces.}  Trans. Amer. Math. Soc. (2017). To appear. \href{http://arxiv.org/abs/1602.07585}{arXiv/1602.07585}



%\bibitem{DSTG01} M. Dumnicki, T. Szemberg, and H. Tutaj-Gasi\'{n}ska. {\em Counterexamples to the $I^{(3)} \subseteq I^2$ containment.} \\ J. Algebra 393 (2013) pp.24-29. \href{http://arxiv.org/abs/1301.7440}{arXiv/1301.7440}

\bibitem{ELS}  L. Ein, R. Lazarsfeld, and K. Smith. {\em Uniform bounds and  symbolic powers on smooth varieties}.  Invent. Math. \textbf{144} (2001), pp. 241-252.

 %\bibitem{Eisenbud1} D. Eisenbud. {\em Commutative Algebra with a view towards Algebraic Geometry.} Graduate Texts in Math. \textbf{150}, Springer-Verlag, New York, 1995.
  
%\bibitem{Eisenbud-Hochster79} D. Eisenbud and M. Hochster. {\em A Nullstellensatz with Nilpotents and Zariski's Main Lemma on Holomorphic Functions.} J. Algebra \textbf{58} (1979), pp. 157-161.

%\bibitem{Phylo000} N. Eriksson, K. Ranestad, B.  Sturmfels, and S. Sullivant. {\em Phylogenetic algebraic geometry.} Projective Varieties with Unexpected Properties, edited by C. Ciliberto, et al.,  Walter de Gruyter, Berlin, 2005  \href{http://arxiv.org/abs/math/0407033}{arXiv preprint}

%\bibitem{GFEG-000} G. Favacchio and E. Guardo. {\em The minimal free resolution of fat almost complete intersections in $\P^1 \times \P^1$}. \href{http://arxiv.org/abs/1605.04769}{arXiv preprint} 

%\bibitem{LFPMYX000} Louiza Fouli, Paolo Mantero, Yu Xie. {\em Chudnovsky's Conjecture for very general points in $\P^N_k$.} \href{http://arxiv.org/abs/1604.02217}{arXiv preprint}


\bibitem{fossum} R.M. Fossum. {\em The divisor class group of a Krull domain} (1973), Ergebnisse der Mathematik und ihrer Grenzgebiete, vol 74. Springer, Berlin, Heidelberg. 


\bibitem{introtoric}  W. Fulton. {\em Introduction to Toric Varieties} (1993), Annals of Math. Studies 131. Princeton University Press, Princeton, NJ. 
  
%\bibitem{FulSturm0}  W. Fulton and B. Sturmfels. {\em Intersection Theory on Toric Varieties}.  Topology \textbf{36} (1997), pp. 335-353.

\bibitem{GotoWatanabeGraded1} S. Goto and K. Watanabe. {\em On graded rings, I.} J. Math. Soc. Japan \textbf{30} (1978), no. 2, 179--213.

\bibitem{M2} D.R. Grayson and M.E. Stillman. Macaulay 2, a software system for research in algebraic geometry. Available at http://www.uiuc.edu/Macaulay2/.


\bibitem{GrifoHun00} E. Grifo and  C. Huneke. {\em Symbolic powers of ideals defining F-pure and strongly F-regular rings.}  \href{https://arxiv.org/abs/1702.06876}{arXiv/1702.06876}  


%\bibitem{EGAIV} A. Grothendieck and J. Dieudonne, Elements de geometrie algebrique IV, Etude locale des schemas et des morphismes de schemas, Seconde partie, Publications Mathematiques de l`IHES.

 %  \bibitem{harb-hun}  B. Harbourne and C. Huneke. {\em Are symbolic powers highly evolved?}  J. Ramanujan Math. Soc. \textbf{28A} (2013), pp. 247-266. \href{http://arxiv.org/abs/1103.5809}{arXiv preprint}


%\bibitem{HNTTBinomial} H.T. H\`{a}, H.D. Nguyen, N.V. Trung, T.N. Trung. {\em Symbolic Powers of Sums of Ideals.} \href{https://arxiv.org/abs/1702.01766}{arXiv preprint}


%   \bibitem{resurge2} B. Harbourne and A. Seceleanu. {\em Containment counterexamples for ideals of various configurations of points in $\mathbb{P}^n$.} J. Pure Appl. Algebra \textbf{219} (2015), no.4, pp. 1062-1072. \href{http://arxiv.org/abs/1306.3668}{arXiv preprint}

%\bibitem{Hartsh0} R. Hartshorne.  {\em Algebraic Geometry}.  Graduate Texts in Math. \textbf{52}, Springer-Verlag, New York, 1977. 
   
%\bibitem{Hartsh1} R. Hartshorne.  {\em Generalized divisors on Gorenstein schemes}.  K-Theory \textbf{8} (1994), pp. 287-339.
   
%\bibitem{hoch1} M. Hochster.  {\em Rings of invariants of tori, Cohen-Macaulay rings generated by monomials, and polytopes}.  Ann. of Math (2) \textbf{96} (1972), pp. 318-337.
   
%\bibitem{hoch2} M. Hochster. {\em Math 615 Winter 2007 Lecture} 4/6/07.  \href{http://www.math.lsa.umich.edu/~hochster/615W07/L04.06.pdf}{Online link.}
  
\bibitem{HH1}  M. Hochster and C. Huneke. {\em Comparison of ordinary and symbolic powers of ideals}.  Invent. Math. \textbf{147} (2002), pp. 349-369.
  
%\bibitem{HH2}  M. Hochster and C. Huneke. {\em Fine behavior of symbolic powers of ideals}.   Illinois J. Math. \textbf{51} (2007), no. 1, pp. 171--183. 

%\bibitem{hun1}  C. Huneke. {\em Uniform bounds in Noetherian rings.}  Invent. Math. \textbf{107} (1992), pp. 203-223.
    
\bibitem{HKV}  C. Huneke, D. Katz, and J. Validashti. {\em Uniform equivalence of symbolic and adic topologies}. Illinois J. Math. \textbf{53} (2009), no. 1, pp. 325--338. 
    
\bibitem{HKV2}  C. Huneke, D. Katz, and J. Validashti. {\em Uniform symbolic topologies and finite extensions}. J. Pure Appl. Algebra \textbf{219} (2015), no. 3, pp. 543--550.

\bibitem{HL02}  C. Huneke and G. J. Leuschke. {\em Two theorems about maximal Cohen-Macaulay modules}. Math. Ann \textbf{324} (2002), no. 2, pp. 391--404.
     
%\bibitem{MRJ1} M.R. Johnson. {\em Containing symbolic powers in regular rings.} Comm. in Algebra Vol. \textbf{42} (2014), no. 8, pp. 3552--3557.
     
%\bibitem{Lip0}  J. Lipman. {\em Rational singularities, with applications to algebraic surfaces and unique factorization}. Inst. Hautes \'{E}tudes Sci. Publ. Math.  (1969), no. 36, pp. 195--279.
     
%\bibitem{Lip78}  J. Lipman. {\em Desingularization of two-dimensional schemes}. Ann. Math.  (1978), \textbf{107} no. 2, pp. 151--207.

\bibitem{MaSchwede17} L. Ma and K. Schwede.  {\em Perfectoid multiplier/test ideals in regular rings and bounds on symbolic powers.}  \href{https://arxiv.org/abs/1705.02300}{arXiv/1705.02300} 

%\bibitem{Matsumura}  H. Matsumura. {\em Commutative Ring Theory.} Cambridge Univ. Press, Cambridge, MA (1989).  

%\bibitem{JSMilneAG} J.S. Milne. {\em Algebraic Geometry.} Version 5.22, 2013.  \href{http://www.jmilne.org/math/CourseNotes/ag.html}{http://www.jmilne.org/math/CourseNotes/ag.html}

%\bibitem{AAMore13} A.A. More. {\em Uniform bounds on symbolic powers.} J. Algebra \textbf{383} (2013), pp. 29-41. 

%\bibitem{Murthy1}  M.P. Murthy. {\em Vector bundles over affine surfaces birationally equivalent to a ruled surface.} Ann. of Math (2) \textbf{89} (1969), pp. 242-253. 

\bibitem{Raicu1011} C. Raicu. {\em Secant Varieties of Segre--Veronese Varieties.} Algebra \& Number Theory 6, no. 8:1817-1868, 2012.  \href{https://arxiv.org/abs/1011.5867}{arXiv/1011.5867} 

%\bibitem{Roberts0} P.C. Roberts. {\em Multiplicities and Chern Classes in Local Algebra.} Cambridge Univ. Press, Cambridge, (1998).


\bibitem{Schreyer1707} F.-O. Schreyer. {\em Horrocks splitting on Segre-Veronese varieties.} \href{https://arxiv.org/abs/1707.00249}{arXiv/1707.00249} 


%\bibitem{SidSull001} J. Sidman and S. Sullivant. {\em Prolongations and Computational Algebra.} Canad. J. Math \textbf{61} (2009), pp. 930-949 \href{http://arxiv.org/abs/math/0611696}{arXiv preprint}

%\bibitem{Sta17} The Stacks Project authors. {\em The Stacks Project}, 2017. 

%\bibitem{BStSSull001} B. Sturmfels and S. Sullivant. {\em Combinatorial Secant Varieties.} Quarterly Journal of Pure and Applied Mathematics \textbf{2} (2006), pp. 285-309  (Special issue in honor of Robert Macpherson) \href{http://arxiv.org/abs/math/0506223}{arXiv preprint}

%\bibitem{Sull001} S. Sullivant. {\em Combinatorial Symbolic Powers.} J. Algebra 319 (2008), \textbf{no. 1}, pp. 115--142. \href{http://arxiv.org/abs/math/0608542}{arXiv preprint} 

%\bibitem{i-swanson} I. Swanson. {\em Linear equivalence of topologies.} Math.Z. \textbf{234} (2000), no.4, pp. 755-775.
   
%\bibitem{TaYo1} S. Takagi, and K. Yoshida. {\em Generalized test ideals and symbolic powers.}  Michigan Math. J. \textbf{57} (2008), pp. 711--725. 

%\bibitem{SinghSpiroff000} A.K. Singh and S. Spiroff. {\em Divisor class groups of graded hypersurfaces}. Contemporary Mathematics \textbf{448} (2007) 237-243.

%\bibitem{SzemSzpo01} T. Szemberg, J. Szpond. {\em On the containment problem.} \href{http://arxiv.org/abs/1601.01308}{arXiv preprint}

\bibitem{Sin05} A.K. Singh. {\em The F-Signature of an affine semigroup ring}. J. Pure Appl. Algebra \textbf{196} (2005), no. 2-3, pp. 313--321.

\bibitem{KT11} K. Tucker. {\em F-Signature Exists}. Invent. Math. \textbf{190} (2012), no. 3, pp. 743--765.  \href{https://arxiv.org/pdf/1103.4173.pdf}{arXiv/1103.4173}

\bibitem{MVKorff} M. Von Korff. {\em The $F$-Signature of Toric Varieties.} Dissertation.  \href{http://search.proquest.com/docview/1149994294}{ProQuest Link}  \href{http://arxiv.org/abs/1110.0552}{arXiv/1110.0552}  

\bibitem{Walker001} R.M. Walker. {\em Rational Singularities and Uniform Symbolic Topologies.} To appear in Illinois J. Math (2017). \href{http://arxiv.org/abs/1510.02993}{arXiv/1510.02993} 

\bibitem{Walker002} R.M. Walker. {\em Uniform Harbourne-Huneke Bounds via Flat Extensions.}  \href{https://arxiv.org/abs/1608.02320}{arXiv/1608.02320}

\bibitem{Walker003} R.M. Walker. {\em Uniform Symbolic Topologies via Multinomial Expansions.}  \href{https://arxiv.org/abs/1703.04530}{arXiv/1703.04530}

\bibitem{WY04} K.-i. Watanabe and K.-i. Yoshida. {\em Minimal relative Hilbert-Kunz multiplicity.} Illinois J. Math. \textbf{48} (2004), no.1, pp. 273-294. \href{https://projecteuclid.org/euclid.ijm/1258136184}{Project Euclid link} 

\bibitem{Yao06} Y. Yao. {\em Observations on the F-signature of local rings of characteristic $p$.} J. Algebra \textbf{299} (2006), no.1, pp. 198-218.
  
\end{thebibliography}
\end{document}